\documentclass[12pt,a4paper,reqno]{amsart}

\usepackage{a4wide}
\usepackage{blindtext}
\usepackage[T1]{fontenc}
\usepackage[utf8]{inputenc}
\usepackage{amsmath,amsthm,amsfonts,amssymb,amscd}
\usepackage{dsfont}
\usepackage[shortlabels]{enumitem}
\usepackage{bigints}
\usepackage{pst-node}
\usepackage{tikz-cd}
\usepackage[final]{pdfpages}
\usepackage[hidelinks]{hyperref}
\newtheorem{thm}{Theorem}[section]
\newtheorem{thmx}{Theorem}

\newtheorem{lemma}[thm]{Lemma}

\theoremstyle{remark}
\newtheorem{rmk}[thm]{Remark}

\newcommand{\D}{\Delta}
\newcommand{\qand}{\quad\text{and}\quad}

\DeclareMathOperator{\corr}{Cor}
\DeclareMathOperator{\ld}{LD}
\DeclareMathOperator{\mld}{MLD}
\DeclareMathOperator{\leb}{Leb}

\numberwithin{equation}{section}

\thanks{JFA and JSM are partially supported by  CMUP (UID/MAT/00144/2019), PTDC/MAT-PUR/28177/2017 and   PTDC/MAT-PUR/4048/2021, which are funded by FCT (Portugal) with national (MEC) and European structural funds through the program  FEDER, under the partnership agreement PT2020. JSM is also supported by  the FCT doctoral scholarship 2021.07090.BD}
\keywords{Decay of correlations; Large deviations; Maximal large deviations; Young structures; Recurrence rates; Partial hyperbolicity}
\subjclass[2010]{37A05, 37A25, 37D25,37D30}

\begin{document}
\title[From Decay of Correlations to Recurrence Times]{From Decay of Correlations to Recurrence Times\\ in  Systems with Contracting Directions}

\author[J. F. Alves]{Jos\'{e} F. Alves}
\address{Jos\'{e} F. Alves\\ Centro de Matem\'{a}tica da Universidade do Porto\\ Rua do Campo Alegre 687\\ 4169-007 Porto\\ Portugal.}
\email{jfalves@fc.up.pt} \urladdr{http://www.fc.up.pt/cmup/jfalves}

\author[J. S. Matias]{Jo\~{a}o S. Matias}
\address{Jo\~{a}o S. Matias\\
Centro de Matem\'{a}tica da Universidade do Porto\\ Rua do Campo Alegre 687\\ 4169-007 Porto\\ Portugal.}
\email{up201504959@fc.up.pt} 

\maketitle



\begin{abstract}

Classical results due to L.-S. Young establish that, for dynamical systems admitting suitable inducing schemes, the rate of decay of correlations can be derived from the recurrence properties of the inducing structure. Conversely, in the context of non-invertible systems -- those lacking contracting directions -- converse statements have been obtained, whereby decay rates imply recurrence estimates. In the present work, we extend these converse results to the setting of invertible systems, which inherently involve contracting directions.

\end{abstract}


\tableofcontents

\section{Introduction}


In recent decades, the analysis of statistical properties has assumed an increasingly prominent role within the broader study of Dynamical Systems. Rather than focusing solely on individual orbits or deterministic trajectories, this perspective seeks to understand the typical behaviour of large sets of orbits over time. Central to this approach is the notion of \emph{physical measures}, which describe the asymptotic time averages of observables along sets of orbits having positive Lebesgue measure. These measures provide a meaningful link between the abstract dynamical framework and quantities that are observable in practice. A particularly significant subclass of physical measures is given by the so-called \emph{Sinai-Ruelle-Bowen (SRB) measures}. These are ergodic invariant probability measures characterised by the presence of at least one positive Lyapunov exponent and the property that their conditional measures along local unstable manifolds are absolutely continuous with respect to the induced Riemannian volume. 

Although SRB  measures yield laws of large numbers for dynamically defined processes, the mixing property often proves useful for investigating finer statistical features. A discrete-time dynamical system~$f$ is said to be \emph{mixing} with respect to an invariant probability measure $\mu$ if
\begin{equation*}
\left|\mu\left(f^{-n}(A)\cap B\right)-\mu(A)\mu(B)\right|\rightarrow 0, \quad\text{as  } n\rightarrow +\infty,
\end{equation*}
for any measurable sets $A$ and $B$. Defining the \emph{correlation function} for observables $\varphi, \psi\colon M\to \mathbb{R}$,
\begin{equation}\label{cor}
\corr_{\mu}(\varphi, \psi\circ f^n) = \left|\int \varphi (
\psi\circ f^n) \, d\mu - \int \varphi \, d\mu \int \psi \, d\mu\right|,
\end{equation}
it is sometimes possible to obtain explicit decay rates,  provided the observables $\varphi$ and $\psi$ possess sufficient regularity.
Note that,   taking the observables $\varphi$ and $\psi$ as characteristic functions of measurable sets, equation~\eqref{cor} corresponds to the definition of mixing. 

Another important observation is that a mixing system is necessarily ergodic. In particular, for such systems, Birkhoff's Ergodic Theorem guarantees that, for any integrable observable $\varphi$,
\begin{equation*}\label{BK}
\lim_{n\to+\infty} \frac{1}{n}
\sum_{j=0}^{n-1} \varphi\circ f^j = \int \varphi d\mu
\end{equation*}
holds  $\mu$-almost everywhere. This naturally raises the interesting question of how quickly this convergence takes place, that is, how fast the \emph{large deviation} at time~$n$,
\begin{equation*}
\ld_{\mu}(\varphi,\epsilon,n)=\mu\left(\left|\frac{1}{n}\sum_{i=0}^{n-1}\varphi\circ f^i - \int\varphi d\mu\right|>\epsilon\right),
\end{equation*}
tends to zero as $n\to\infty$.
One may also consider the \emph{maximal large deviation} at time $n$,
\begin{equation*}\label{mld}
\mld_{\mu}(\varphi,\epsilon,n)=\mu\left(\sup_{j\geqslant n}\left|\frac{1}{j}\sum_{i=0}^{j-1}\varphi\circ f^i-\int\varphi d\mu\right|>\epsilon\right).
\end{equation*}
Clearly, the decay rate of $\mld_{\mu}(\varphi,\epsilon,n)$ implies a corresponding decay rate for $\ld_{\mu}(\varphi,\epsilon,n)$. The converse also holds in the exponential and stretched exponential cases -- see~\eqref{eq.sea} in Subsection~\ref{se.exp} -- but fails in the polynomial case.

Significant advances in understanding the statistical properties of dynamical systems were made by Sinai, Ruelle and Bowen in the 1970s, when Markov partitions were introduced to analyse the statistical behaviour of uniformly hyperbolic systems — Axiom~A attractors — via symbolic dynamics; see \cite{Bo75,BR75,Ru76,Si72}. Since then, a central objective in dynamical systems theory has been to extend this strategy to much wider classes of systems, including those exhibiting nonuniform hyperbolic behaviour.

A major breakthrough came with the work of Young \cite{Y98,Y99}, who developed the framework of \emph{Young towers}, that is, Markov extensions of the original system constructed via \emph{inducing schemes}. This abstract structure has proven extremely effective in the systematic study of a broad range of nonuniformly hyperbolic dynamical systems, including piecewise hyperbolic maps, billiards with convex scatterers, logistic maps, intermittent maps, and Hénon-like attractors — the latter studied in collaboration with Benedicks in~\cite{BY00}. Young towers are built from inducing schemes that isolate a suitable subset of the phase space (a \emph{Young structure}) where the dynamics display sufficiently strong hyperbolic behaviour; see subsections~\ref{se.young} and~\ref{se.towers} below for details. Points in this subset return under iteration after variable amounts of time, and these return times define the levels of the tower.

Although the return times may be unbounded, the construction ensures that the set of points with large return times has small measure, allowing one to control the statistical behaviour of the entire system.
 The flexibility of allowing arbitrarily large return times   is essential for treating nonuniformly hyperbolic systems. 
%
%
%
In some cases, one can explicitly determine the asymptotic behaviour of the return time function, which quantifies how rapidly typical orbits exhibit strong hyperbolic behaviour. These estimates are instrumental in deriving precise statistical laws for the system. For dynamical systems that admit a Young tower, a broad array of such laws has been established, including decay of correlations and the central limit theorem \cite{Y99}, large deviations \cite{AF19, MN05, MN08}, and the vector-valued almost sure invariance principle \cite{MN09}.

Since Young towers have become one of the most powerful tools for analysing statistical properties of nonuniformly hyperbolic systems, a natural question is to determine the extent to which such structures exist. In \cite{ALP05, G06}, the authors demonstrated that hyperbolic times can be employed to construct Young towers, and explored how the tail of hyperbolic times relates to the tail of recurrence times in the context of nonuniformly expanding systems. 
For partially hyperbolic diffeomorphisms whose central direction is mostly expanding, this question was addressed in \cite{AL15, AP10}, allowing --~together with results from \cite{AP08, M09, MN09,Y98,Y99}~-- the derivation of several statistical properties for the SRB measures obtained in \cite{ABV00}.
Although no general results comparable to \cite{AL15, AP10} exist for partially hyperbolic systems with mostly contracting central directions, Young structures have been constructed in specific cases by Castro \cite{C02, C04}, enabling the derivation of statistical properties such as decay of correlations and the Central Limit Theorem for the SRB measures obtained in \cite{BV00}; see also Dolgopyat's work \cite{D04} in the three-dimensional setting.

The connection between Young towers and statistical properties was further developed in \cite{ADL13, ADL17}, where the authors characterised the measures that can be lifted to Young towers, both for systems with and without contracting directions.
A further step was taken in~\cite{AFL11}, where bidirectional connections were established between rates for decay of correlations, large deviations and hyperbolic times. This approach enabled the use of results from \cite{ALP05, G06} to prove a partial converse to Young's theorem for nonuniformly expanding systems (noninvertible systems).
More recently, \cite{BS23} improved upon the polynomial case by employing maximal large deviations instead of standard large deviations, yielding sharper estimates for the tails of hyperbolic times.


\emph{A key difficulty preventing the extension of the results in \cite{AFL11} to systems with contracting directions  lies in the fact that, for such systems, decay of correlations cannot be established for H\"older observables against those in $L^\infty$. Instead, both observables must satisfy stronger regularity conditions. In this work, we provide a first contribution toward establishing a bidirectional relationship between decay of correlations and recurrence times associated with Young structures in the setting of invertible systems with contractive directions.}
 
Throughout this text, we consider $M$ to be a finite-dimensional compact Riemannian manifold, and $f \colon M \to M$ a piecewise $C^1$ diffeomorphism, possibly with critical, singular, or discontinuity points. Let $d$ denote the distance on $M$, and let $\leb$ represent the Lebesgue measure on the Borel subsets of $M$, both induced by the Riemannian metric. Given a submanifold $\gamma \subset M$, we denote by $\leb_\gamma$ the Lebesgue measure on $\gamma$, induced by the restriction of the Riemannian metric to $\gamma$.
 
 \subsection{Young structures}\label{se.young}

Here, we introduce the fundamental concept of a Young structure, along with several related objects.
We begin by considering continuous families $\Gamma^s$ and $\Gamma^u$ of $C^1$ disks in $M$.
We impose a form of weak forward contraction on the disks in $\Gamma^s$ (referred to as stable disks), and a corresponding weak backward contraction on the disks in $\Gamma^u$ (unstable disks). We say that $\Gamma$ is a \emph{continuous family} of $C^1$ disks in $M$ if there are a  compact metric space $K$, a unit disk $D$ in some $\mathbb{R}^k$ and an injective continuous function $\Phi\colon K\times D\to M$ such that
\begin{itemize}
	\item $\Gamma=\left\{\Phi\left(\{x\}\times D\right)\colon x\in K\right\}$;
	\item$\Phi$ maps $K\times D$ homeomorphically onto its image;
	\item $x\mapsto\Phi|_{\{x\}\times D}$ defines a continuous map from $K$ into $\text{Emb}^1(D,M)$, where $\text{Emb}^1(D,M)$ denotes the space of $C^1$ embeddings of $D$ into $M$.  
\end{itemize}
All the disks in $\Gamma$ have the same dimension $\dim D$, which we also denote by $\dim \Gamma$.  
We say that a compact set $\Lambda \subset M$ has a \emph{product structure} if there exist continuous families of $C^1$ disks $\Gamma^s$ and $\Gamma^u$ such that
\begin{itemize}
	\item $\Lambda = \left(\bigcup_{\gamma \in \Gamma^s} \gamma \right) \cap \left(\bigcup_{\gamma \in \Gamma^u} \gamma \right)$;
	\item $\dim \Gamma^s + \dim \Gamma^u = \dim M$;
	\item each $\gamma \in \Gamma^s$ meets each $\gamma \in \Gamma^u$ in exactly one point.
\end{itemize}
We say that $\Lambda_0 \subset \Lambda$ is an \emph{s-subset} if $\Lambda_0$ has a product structure with respect to families $\Gamma_0^s$ and $\Gamma_0^u$ such that $\Gamma_0^s \subset \Gamma^s$ and $\Gamma_0^u = \Gamma^u$; \emph{u-subsets} are defined similarly.  
Let $\gamma^*(x)$ denote the disk in $\Gamma^*$ containing $x \in \Lambda$, for $* = s,u$.  
Consider the \emph{holonomy map} $\Theta_{\gamma,\gamma'} \colon \gamma \cap \Lambda \to \gamma' \cap \Lambda$, defined for each $x \in \gamma \cap \Lambda$ by
\[
\Theta_{\gamma,\gamma'}(x) = \gamma^s(x) \cap \gamma'.
\]
We say that a compact set $\Lambda$ has a \emph{Young structure} (with respect to $f$) if $\Lambda$ has a product structure given by continuous families of $C^1$ disks $\Gamma^s$ and $\Gamma^u$ such that
\[
\leb_\gamma \big( \Lambda \cap \gamma \big) > 0, \quad \text{for all } \gamma \in \Gamma^u,
\]
and conditions \ref{Y1}--\ref{Y5} below are satisfied.

\begin{enumerate}[label=\textbf{(Y$_{\arabic*}$)}]
	\item \label{Y1} \textbf{Markov:} there are pairwise disjoint s-subsets $\Lambda_1$, $\Lambda_2$, $\cdots$$\subset$ $\Lambda$ such that
	\begin{itemize}
		\item $\leb_\gamma\left(\left(\Lambda \setminus \cup_i\Lambda_i\right)\cap\gamma\right)=0$ for all $\gamma\in\Gamma^u$;
		\item for each $i\geqslant1$, there is $R_i\in\mathbb{N}$ such that $f^{R_i}\left(\Lambda_i\right)$ is a u-subset and, moreover, for all $x\in\Lambda_i,$
		\begin{equation*}
			f^{R_i}\left(\gamma^s(x)\right)\subset \gamma^s\left(f^{R_i}(x)\right) \text{ and } 	f^{R_i}\left(\gamma^u(x)\right)\supset \gamma^u\left(f^{R_i}(x)\right).
		\end{equation*}
	\end{itemize}
\end{enumerate}

This Markov property allows us to introduce a \emph{recurrence time} ${R \colon \Lambda \to \mathbb{N}}$ and a \emph{return map} $f^R \colon \Lambda \to \Lambda$, defined for each $i \geqslant 1$ by
\begin{equation*}
	R|_{\Lambda_i} = R_i \quad \text{and} \quad f^R|_{\Lambda_i} = f^{R_i}|_{\Lambda_i}.
\end{equation*}
We remark that $R$ and $f^R$ are defined on a full $\leb_\gamma$-measure subset of $\Lambda \cap \gamma$, for each $\gamma \in \Gamma^u$. Thus, there exists a set $\Lambda' \subset \Lambda$ intersecting each $\gamma \in \Gamma^u$ in a full $\leb_\gamma$-measure subset, such that $\left(f^R\right)^n(x)$ belongs to some $\Lambda_i$ for all $n \geqslant 0$ and $x \in \Lambda'$.  
For points $x, y \in \Lambda'$, we define the \emph{separation time}
\begin{equation}\label{eq.separa}
	s(x,y) = \min \left\{ n \geqslant 0 \colon \left(f^R\right)^n(x) \text{ and } \left(f^R\right)^n(y) \text{ lie in distinct } \Lambda_i \text{'s} \right\},
\end{equation}
with the convention that $\min(\emptyset) = \infty$. For definiteness, we set the separation time equal to zero for all other points.  
For the remaining conditions, we consider constants $C > 0$ and $0 < \beta < 1$ depending only on $f$ and $\Lambda$.

\begin{enumerate}[label=\textbf{(Y$_{\arabic*}$)}]
	\setcounter{enumi}{1}
	\item \label{Y2} \textbf{Contraction on stable disks:} for all $i\geqslant1$, $\gamma\in\Gamma^s$ and $x,y\in\gamma\cap\Lambda_i$,
	\begin{itemize}
		\item $d\left(\left(f^R\right)^n(x),\left(f^R\right)^n(y)\right)\leqslant C\beta^n$, for all $n\geqslant0$;
		\item$d\left(f^j(x),f^j(y)\right)\leqslant Cd(x,y)$, for all $1\leqslant j \leqslant R_i$.	
	\end{itemize}
	
\end{enumerate}

\begin{enumerate}[label=\textbf{(Y$_{\arabic*}$)}]
	\setcounter{enumi}{2}
	\item \label{Y3} \textbf{Expansion on unstable disks}: for all $i\geqslant1$, $\gamma\in\Gamma^u$ and $x,y\in\gamma\cap\Lambda_i$,
	\begin{itemize}
		\item $d\left(\left(f^R\right)^n(x),\left(f^R\right)^n(y)\right)\leqslant C\beta^{s(x,y)-n}$, for all $n\geqslant0$;
		\item$d\left(f^j(x),f^j(y)\right)\leqslant Cd\left(f^R(x),f^R(y)\right)$, for all $1\leqslant j \leqslant R_i$.	
	\end{itemize}

	\item \label{Y4}\textbf{ Gibbs:} for all $i\geqslant 1$, $\gamma\in\Gamma^u$ and $x,y\in\gamma\cap\Lambda_i$,
	\begin{equation*}
		\log\frac{\det Df^R|_{T_x\gamma}}{\det Df^R|_{T_y\gamma}}\leqslant C\beta^{s\left(f^R(x),f^R(y)\right)}.
	\end{equation*}
	
	\item \label{Y5}\textbf{ Regularity of the stable holonomy}: for all $\gamma,\gamma'\in\Gamma^u$, the measure $\left(\Theta_{\gamma,\gamma'}\right)_* \leb_\gamma$ is absolutely continuous with respect to $\leb_{\gamma'}$ and its density $\rho_{\gamma,\gamma'}$ satisfies\begin{equation*}
		\frac{1}{C}\leqslant\int_{\gamma'\cap\Lambda}\rho_{\gamma,\gamma'} d\leb_{\gamma'}\leqslant C \quad\text{and}\quad\log\frac{\rho_{\gamma,\gamma'}(x)}{\rho_{\gamma,\gamma'}(y)}\leqslant C\beta^{s(x,y)},
	\end{equation*}
	for all $x,y\in\gamma'\cap\Lambda$.
\end{enumerate}
We define the \emph{tail of recurrence times} as the set
\begin{equation*}
	\left\{R>n\right\}:=\left\{x\in\Lambda\colon R(x)>n\right\}.
\end{equation*}

\subsection{Partially hyperbolic systems}\label{phs}
We say that a compact set $K \subset M$ is a \emph{partially hyperbolic set} for $f$ if it is forward invariant under $f$ and there exists a $Df$-invariant splitting 
\[
T_K M = E^{ss} \oplus E^{cu},
\]
together with a constant $0 < \lambda < 1$, such that for some choice of Riemannian metric on~$M$,
\begin{enumerate}
    \item $E^{ss}$ is \emph{uniformly contracting}: $\lVert Df|_{E_x^{ss}} \rVert < \lambda$, for all $x \in K$;
    \item $E^{cu}$ is \emph{dominated} by $E^{ss}$: 
    \(
    \lVert Df|_{E_x^{ss}} \rVert \cdot \lVert Df^{-1}|_{E_{f(x)}^{cu}} \rVert < \lambda\) for all \(x \in K.
    \)
\end{enumerate}
We refer to $E^{ss}$ as the \emph{strong-stable} subbundle and to $E^{cu}$ as the \emph{centre-unstable} subbundle.

One of the main objectives of this work is to obtain converse results to those of Young in the context of partial hyperbolicity -- specifically, to show that for certain mixing rates, the decay of correlations is sufficient to guarantee the existence of Young structures with corresponding recurrence tails.
Recall that an observable $\varphi \colon M \to \mathbb{R}$ is said to be $\eta$-H\"older continuous, for some $\eta > 0$, if there exists a constant $C > 0$ such that
\[
\left| \varphi(x) - \varphi(y) \right| \leqslant C d(x, y)^\eta, \quad \forall x, y \in M.
\]
Let $\mathcal{H}_\eta(M)$ denote the space of all $\eta$-H\"older continuous observables. The H\"older norm of an observable $\varphi \in \mathcal{H}_\eta$ is defined as
\[
\lVert \varphi \rVert_{\mathcal{H}_\eta} := \lVert \varphi \rVert_\infty + \sup_{x \neq y} \frac{\left| \varphi(x) - \varphi(y) \right|}{d(x, y)^\eta}.
\]
In our first result, we establish a connection between maximal large deviations for H\"older observables and the existence of Young structures with associated recurrence tails.

We use the notation $\lesssim$ to indicate inequality up to a constant factor depending only on the dynamics. Furthermore, we refer to the decay as \emph{exponential}, \emph{stretched exponential}, or \emph{polynomial} if it satisfies
$
\lesssim \epsilon^{-\tau n}$,  $ \lesssim \epsilon^{-\tau n^\theta}$, or $\lesssim n^{-\alpha},
$
respectively, for some constants $\tau, \alpha > 0$ and $0 < \theta < 1$.

\begin{thmx}\label{main2}
	Let $f\colon M\to M$ be a $C^{1+\eta}$ diffeomorphism and let $K\subset M$ be a  partially hyperbolic set. Suppose that $f$ admits an ergodic $SRB$ measure $\mu$ supported on $K$ with all Lyapunov exponents along $E^{cu}$  positive.  
	\begin{enumerate}[(a)]
		\item If there exists $\alpha>1$ such that $\mld_\mu(\varphi,\epsilon, n)\lesssim n^{-\alpha}$ for every $\varphi\in\mathcal{H}_\eta(M)$, then there exists a Young structure such that $\leb_\gamma\left\{R>n\right\}\lesssim n^{-\alpha}$, for every $\gamma\in\Gamma^u$.
	\end{enumerate}
	\begin{enumerate}[(b)]
		\item If there exist $\tau,\theta>0$ such that $\mld_\mu(\varphi,\epsilon, n)\lesssim n^{-\alpha}$ for every $\varphi\in\mathcal{H}_\eta(M)$, then there exists a Young structure and $\tau'>0$ such that $\leb_\gamma\left\{R>n\right\}\lesssim e^{-{\tau'} n^{\theta}}$, for every $\gamma\in\Gamma^u$.
	\end{enumerate}
\end{thmx}

The proof of Theorem~\ref{main2} is presented in Section~\ref{proof2}.  
While the original result in~\cite{AFL11} relies on large deviations,   the stronger condition of \emph{maximal large deviations} introduced in~\cite{BS23}  yields  sharper estimates in the polynomial case.  
In the (stretched) exponential case, however, it makes no essential difference whether one uses large deviations or maximal large deviations.  
It is worth noting that in Theorem~\ref{main2}, the maximal large deviations assumption is only required for the very specific type of H\"older observables  in~\eqref{eq.phin}.

Examples of systems exhibiting partially hyperbolic sets in the conditions of Theorem~\ref{main2} include the \emph{derived-from-Anosov} systems discussed in \cite[Appendix~A]{ABV00}, which demonstrate exponential mixing rates, as well as the \emph{solenoid with intermittency} presented in \cite[Section~5]{AP08a}, where polynomial mixing rates are observed.

\subsection{Young towers}\label{se.towers}

Here, we introduce the tower map associated with the return map of a Young structure. 
Consider a set $\Lambda \subset M$ equipped with a Young structure, and the corresponding return map $f^R \colon \Lambda \to \Lambda$, as defined in Subsection~\ref{se.young}. We associate to these objects the \emph{tower}
\begin{equation*}
	{\D}=\left\{(x,\ell)\colon x\in\Lambda\text{ and } 0\leqslant \ell<R(x)\right\},
\end{equation*}
and the \emph{tower map} ${T}\colon {\D}\to {\D}$, given by
\begin{equation}\label{towermap}
	{T}(x,\ell)=\begin{cases}
		(x,\ell+1), &\text{  if } \ell<R(x)-1\\
		\left(f^R(x),0\right),&\text{ if } \ell=R(x)-1
	\end{cases}.
\end{equation}
The base ${\D}_0$ of the tower ${\D}$ is naturally identified with the set $\Lambda$, and each \emph{level} ${\D}_\ell$ with the set $\left\{ R > \ell \right\} \subset \Lambda$. Since each level of the tower is identified with a subset of~$\Lambda$, it remains meaningful to speak of an \emph{SRB measure} for $T$, understood as a probability measure whose conditionals on the unstable disks in the family $\Gamma^u$ defining the Young structure are absolutely continuous with respect to the Lebesgue measures on those disks. It is a well-known fact that the tower map admits a unique SRB measure $\nu$, which is ergodic; see \cite[Section~2]{Y98} or \cite[Theorem~4.11]{A20} for further details. The maps $f$ and $T$ are connected via the measurable semiconjugacy
\begin{equation}\label{eq.conjuga}
	\begin{array}{rccc}
		\pi\colon\!\!\!\!& \D &\longrightarrow & M\\
		&(x,\ell)&\longmapsto & f^\ell(x)
	\end{array},
\end{equation}
such that
\begin{equation*}
	\pi \circ T = f \circ \pi.
\end{equation*}
Moreover, if $\nu$ denotes the unique ergodic SRB measure for $T$, then $\mu = \pi_*\nu$ is the unique ergodic SRB measure for $f$ such that $\mu(\Lambda) > 0$; see  \cite[Theorem~4.9]{A20}.

We extend the separation time $s$ associated with $f^R$ to the tower map ${T}$ by setting
\[
s\big((x,\ell),(x',\ell')\big) = 
\begin{cases}
	s(x,x'), & \text{if } \ell = \ell', \\
	0,       & \text{otherwise}.
\end{cases}
\]
Fixing a constant $\beta > 0$ for which conditions \ref{Y2}--\ref{Y5} in Subsection~\ref{se.young} hold, we define the linear space of functions
\[
\mathcal{F}_{\beta}({\D}) = \left\{ \varphi \colon {\D} \to \mathbb{R} \mid \exists C > 0 : \left| \varphi(y) - \varphi(z) \right| \leqslant C \beta^{s(y,z)}, \quad \forall y,z \in {\D} \right\}.
\]
Consider $\mathcal{F}_\beta({\D})$ endowed with the norm
\[
\lVert \varphi \rVert_{\beta} := \left| \varphi \right|_{\beta} + \lVert \varphi \rVert_{\infty}, \quad \text{where} \quad 
\left| \varphi \right|_{\beta} = \sup_{z \neq y} \frac{\left| \varphi(y) - \varphi(z) \right|}{\beta^{s(y,z)}}.
\]
We may interpret $\beta^{s(y,z)}$ as a dynamical distance between the points $y,z\in\Delta$. With this viewpoint, $\mathcal{F}_\beta({\D})$ can be regarded as a space of Lipschitz observables.


\begin{thmx}\label{main3}
	Let $T\colon\D\to\D$ be   a Young tower and $\nu$ be  its SRB measure.  
	\begin{enumerate}[label=(\alph*)]
		\item If there exist $C>0$ and $\alpha>1$ such that $$\corr_{\nu}\left(\psi, \xi\circ {T}^n\right)\leqslant C \lVert\psi\rVert_{\beta} \lVert\xi\rVert_{\beta}n^{-\alpha},$$ for all  ${\psi,\xi\in\mathcal{F}_{\beta}(\D)}$, then, for all $q>\max\{1,\alpha\}$, there exists  $C_{\alpha,q}>0$ only depending on $\alpha$ and $q$ such that 
		$$
		\mld_\nu(\varphi,\epsilon,n)\leqslant C_{\alpha,q}   \lVert\varphi\rVert_{\beta}\lVert\varphi\rVert_{\infty}^{2q-1}\epsilon^{-2q}n^{-\alpha},
$$
 for all $\varphi\in\mathcal{F}_{\beta}(\D)$.
			\item If there exist $C,\tau,\theta>0$ such that 
		$$\corr_{\nu}\left(\psi,\xi\circ {T}^n\right)\leqslant C \lVert\psi\rVert_{\beta} \lVert\xi\rVert_{\beta}e^{-\tau n^\theta},$$ 
		for all ${\psi,\xi\in\mathcal{F}_{\beta}(\D)}$, then  for all  $\varphi\in\mathcal{F}_{\beta}(\D)$ there exist $C_{\tau,\theta,\varphi}>0$ only depending  on $\tau,\theta$ and $\varphi$  such that 
		$$
		\mld_\nu(\varphi,\epsilon,n)\leqslant C_{\tau,\theta,\varphi}e^{-\tau' n^{\theta'}\epsilon^{2\theta'}},
		$$
		where  $\theta'=\theta/(\theta+1)$, 
		$\tau'=C_{\tau,\theta} \lVert\varphi\rVert_{\beta}^{-2\theta'}$ and $C_{\tau,\theta}>0$ depends only    on $\tau,\theta$.
	\end{enumerate}
\end{thmx}


The proof of Theorem~\ref{main3} is presented in Section~\ref{proof1}.  
For dynamical systems admitting a Young tower, we are able to deduce further information on the large deviations of certain observables, provided that these observables satisfy sufficient regularity conditions.  
In the particular case of H\"older continuous observables, Theorem~\ref{main3} can be applied to obtain the following result.

\begin{thmx}\label{Maximal large deviations}
	Let $T\colon\D\to\D$ be a Young tower for 
	$f:M\to M$,  $\nu$ be the SRB measure of~$T$ and $\mu=\pi_*\nu$. Given $0<\eta<1$, set $\beta'=\beta^{\eta/2}$.
	\begin{enumerate}[label=(\alph*)]
		\item
		If there exist $C>0$ and $\alpha>1$ such that $$\corr_{\nu}\left(\psi, \xi\circ {T}^n\right)\leqslant C \lVert\psi\rVert_{\beta'} \lVert\xi\rVert_{\beta'}n^{-\alpha},$$ for all  ${\psi,\xi\in\mathcal{F}_{\beta'}(\D)}$, then, for all $q>\max\{1,\alpha\}$, there exists  $C_{\alpha,q}>0$ only depending on $\alpha$ and $q$ such that $$\mld_\mu(\varphi,\epsilon,n)\leqslant C_{\alpha,q}   \lVert\varphi\rVert_{\mathcal H_\eta} \lVert\varphi\rVert_{\infty}^{2q-1}\epsilon^{-2q}n^{-\alpha},
$$ for all $\varphi\in\mathcal{H}_\eta(M)$.
	\item  If there exist $C,\tau,\theta>0$ such that 
		$$\corr_{\nu}\left(\psi,\xi\circ {T}^n\right)\leqslant C \lVert\psi\rVert_{\beta'} \lVert\xi\rVert_{\beta'}e^{-\tau n^\theta},$$ 
		for all ${\psi,\xi\in\mathcal{F}_{\beta'}(\D)}$, then for all $\varphi\in\mathcal{H}_\eta(M)$ there exist $C_{\tau,\theta,\varphi}>0$ only depending  on $\tau,\theta$ and $\varphi$  such that 
		$$
		\mld_\mu(\varphi,\epsilon,n)\leqslant C_{\tau,\theta,\varphi}e^{-\tau' n^{\theta'}\epsilon^{2\theta'}},
		$$
		where  $\theta'=\theta/(\theta+1)$, 
		$\tau'=C_{\tau,\theta} \lVert\varphi\rVert_{\mathcal H_\eta}^{-2\theta'}$ and $C_{\tau,\theta}>0$ depends  only  on $\tau,\theta$.
	\end{enumerate}
\end{thmx}

The proof of Theorem~\ref{Maximal large deviations} is presented in Section~\ref{lift}.  
As one may observe, these two results are related to those in~\cite{AF19, AFL11, M09}.  
The main distinction in our setting is that we assume decay of correlations for observables on the tower under the condition that \emph{both} observables are Lipschitz, whereas the aforementioned works on non-invertible systems allow one of the observables to be merely bounded.  

Although verifying that an observable is bounded is significantly simpler than verifying that it is Lipschitz, the relevance of Theorems~\ref{main3} and~\ref{Maximal large deviations} lies in their applicability to partially hyperbolic systems. In that context, it is known that decay of correlations between bounded observables is not attainable.

Finally, as a consequence of the preceding results, we obtain the following converse to Young's results for invertible systems, once again within a partially hyperbolic framework.

\begin{thmx}\label{main}
	Let $f\colon M\to M$ be a $C^{1+\eta}$ diffeomorphism and  $K\subset M$ be a  partially hyperbolic set with a Young structure $\Lambda\subset K$. Let  $T:\Delta\to\Delta$ be the  associated tower map and $\nu$ its SRB measure.   
	\begin{enumerate}[label=(\alph*)]
		\item If there exists $\alpha>1$ such that $\corr_{{\nu}}\left(\varphi,\psi\circ {T}^n\right)\lesssim n^{-\alpha}$, for every $\varphi,\psi\in\mathcal{F}_{\beta}({\D})$, then $f$ admits a Young structure with $\leb_\gamma\left\{R>n\right\}\lesssim n^{-\alpha}$, for any $\gamma\in\Gamma^u$.
	\end{enumerate}
	\begin{enumerate}[label=(\alph*)]
		\setcounter{enumi}{1}
		\item If there exist $\tau,\theta>0$ such that $\corr_{{\nu}}\left(\varphi,\psi\circ {T}^n\right)\lesssim e^{-\tau n^\theta}$, for every $\varphi,\psi\in\mathcal{F}_{\beta }({\D})$, then $f$ admits a Young structure with $\leb_\gamma\left\{R>n\right\}\lesssim e^{-\tau' n^{\theta'}}$, with $\theta'=\theta/(\theta+1)$, for some $\tau'>0$  and   any $\gamma\in\Gamma^u$.
	\end{enumerate}
\end{thmx}
The proof of Theorem~\ref{main}   is presented in Section~\ref{proof2}.

\section{Maximal large deviations in the tower}\label{proof1}
Here we prove Theorem \ref{main3}.
Instead of considering directly the observable $\varphi$ we start to deduce the result for a sequence of discretisations  $\left(\varphi_k\right)_{k\in\mathbb{N}}$ of   $\varphi$. For that we make use of \cite[Theorem 3.1]{BS23} and \cite[Theorem~2]{AF19}.  
Finally, after obtaining the desired large deviation estimates for the discretisations, we will show how to use them to obtain the (maximal) large deviations estimates for the original observable $\varphi$.
We first define a quotient tower  that   will be  helpful.

\subsection{Quotient tower}\label{se.quotient}
On the tower $\D$ consider the equivalence relation $\sim$ defined by $x\sim y$ if and only if $y\in\gamma^s(x)$ and consider the quotient space $\bar{\D}:=\D/\sim$. By fixing an unstable disk $\gamma_0\in\Gamma^u$, the representative of each equivalence class is given by the holonomy map $\Theta_{\gamma_0}(x):= \gamma^s(x)\cap\gamma_0$. This way, $\bar{\D}$ may be seen as a subset of $\D$ and, in particular, an observable in $\D$ that is constant along stable disks may be seen as an observable in $\bar\D$, and vice-versa. Moreover, we can define a tower map $\overline{T}$ in a similar way as we have defined in~\eqref{towermap}. Associated to the tower maps $T$ and $\overline T$ we have a measurable semiconjugacy 

\begin{equation*}
	\begin{array}{rccc}
		\Theta\colon\!\!\!\!&\D &\longrightarrow &\bar\D\\
		&(x,\ell)&\longmapsto & \left(\Theta_{\gamma_0}(x),\ell\right)
	\end{array}
\end{equation*}
such that
\begin{equation*}
	\Theta\circ {T}=\overline T\circ\Theta.
\end{equation*}
Regarding the measure theoretical properties of these systems, recall that $T$ admits a unique ergodic SRB measure $\nu$. Moreover, the measure $\bar \nu=\Theta_*\nu$ is the unique ergodic $\overline T$-invariant probability measure absolutely continuous with relation to $\leb_{\gamma_0}$. A  complete description of these towers can be found in \cite{A20,Y98}.

\subsection{Discretisation}

The key to obtain the estimates for $\mld_\nu\left(\varphi,\epsilon,n\right)$ is by considering a discretisation of $\varphi$. We begin by defining it and then present some useful properties. Recall that property \ref{Y1}  gives us a partition of $\Delta_0$, which induces a partition on higher levels of the tower $\D$. Collecting all these partitions we obtain a global partition $\mathcal{Q}$ of $\D$. With that in mind, we define, for each $n\in\mathbb{N}$,
\begin{equation*}
	\mathcal{Q}_n=\bigvee_{j=0}^{n-1}T^{-j}\mathcal{Q}:=\left\{\omega_0\cap\omega_1\cap\cdots\cap \omega_{n-1} | \omega_j\in T^{-j}\mathcal{Q}, \text{  for } 0\leqslant j\leqslant n-1\right\}.
\end{equation*}
Given $k\in\mathbb{N}$ we define $\varphi_k\colon \D\to\mathbb{R}$ to be the  \emph{discretisation} of $\varphi$, defined on each $Q\in \mathcal{Q}_{2k}$ by
\begin{equation}\label{eq.varphika}
	\varphi_k|_Q=\inf\left\{\varphi(x)\colon x\in Q\right\}.
\end{equation}
Note that $\varphi_k$ is constant along stable disks, and so it can be also seen as an observable on the quotient tower $\bar\D$. 
Our first preliminary result is about the convergence of $\varphi_k$ as $k$ goes to infinity. For convenience, set for each $x\in \Delta $ (or $\bar\Delta$)
$$b_{2k}(x):=\#\left\{1\leqslant j\leqslant 2k\colon T(x)\in \D_0\right\}.$$
\begin{lemma}\label{uniconv}
	$\left(\varphi_k\right)_{k\in\mathbb{N}}$ converges pointwise to $\varphi$, when $k\to\infty$.
\end{lemma}
\begin{proof}
	We want to prove that for every $\epsilon>0$ there exists a natural number $p$ such that, for every $k\geqslant p$, $\left|\varphi(x)-\varphi_k(x)\right|<\epsilon$. With that in mind,  fix $k\in\mathbb{N}$ and let $x\in Q\in\mathcal{Q}_{2k}$. Thus, and since $\varphi\in\mathcal{F}_\beta(\D)$,
	\begin{equation}\label{uniform}
		\left|\varphi(x)-\varphi_k(x)\right|=\left|\varphi(x)-\varphi(y)\right|\leqslant C\beta^{s(x,y)},
	\end{equation}
	where $x$ and $y$ belong to the same element $Q$ of the partition $\mathcal{Q}_{2k}$. Moreover, we have  $s(x,y)\geqslant b_{2k}(x).$ This means that we can bound (\ref{uniform}) by $C\beta^{b_{2k}(x)}$. As $k$ goes to infinity, $b_{2k}(x)$ also goes to infinity, so we only need to choose $p$ such that $C\beta^{b_{2p}(x)}<\epsilon$ to obtain the desired.
\end{proof}

Until the end of this subsection we   consider $\varphi_k$ as an observable in $\bar\D$. 
The next results allow us to conclude that the observables in question are regular enough for us to obtain decay of correlations estimates, in other words, we will verify that they belong in $\mathcal{F}_\beta({\bar \D})$, which may be thought of as contained in  $ \mathcal{F}_\beta(\D)$.

\begin{lemma}\label{ita}
For all $k$, we have	$\varphi_k\in\mathcal{F}_\beta({\bar \D})$  and $|\varphi_k|_{\beta}\leqslant  |\varphi|_{\beta}$.
\end{lemma}

\begin{proof}
	We need to show that
	\begin{equation*}
		\frac{\left|\varphi_k(x)-\varphi_k(y)\right|}{\beta^{s(x,y)}} \leqslant  |\varphi|_{\beta},
	\end{equation*}
	for all $x,y\in\bar\D$. If $\left|\varphi_k(x)-\varphi_k(y)\right|=0$, then there is nothing to be checked. On the other hand, if $\left|\varphi_k(x)-\varphi_k(y)\right|\neq0$, by the definition of the discretisation $\varphi_k$ we have that $x$ and $y$ lie on different elements of the partition $\mathcal{Q}_{2k}$, which means that $s(x,y)\leqslant 2k$. Let $Q_x$ and $Q_y$ be the elements in $\mathcal Q_{2k}$ such that $x\in Q_x$ and $ y\in Q_y$.
By definition of $\varphi_k$,  there exist $x_k\in Q_x$ and $y_k\in Q_y$ such that 
 $$\varphi_k(x)=\varphi(x_k)\qand \varphi_k(y)=\varphi(y_k).$$
	Now, observe that $x_k,y_k$ have the same trajectory of $x,y$, respectively, until time $2k$. Since $s(x,y)\leqslant 2k$, it follows that $s(x_k,y_k)\leqslant 2k$ and, necessarily, $s(x_k,y_k)= s(x,y)$.
 Therefore,
	\begin{equation*}
		\frac{\left|\varphi_k(x)-\varphi_k(y)\right|}{\beta^{s(x,y)}}=\frac{\left|\varphi(x_k)-\varphi(y_k)\right|}{\beta^{s(x_k,y_k)}}\leqslant  |\varphi|_{\beta},
	\end{equation*}
thus	proving the result.
\end{proof}

\subsection{Maximal large deviations for the discretisation}
Now we  establish estimates for the maximal large deviations of the discretised observables introduced above, with the estimates depending only on the original observable. This will be used to deduce maximal large deviations for the original observable in Subsection~\ref{sub.original}.

  Given $\varphi \in \mathcal{F}_\beta(  \D)$, consider
a large number $k\in\mathbb N$ and $\varphi_k$ as in~\eqref{eq.varphika}. By Lemma \ref{ita}, we have $\varphi_k\in\mathcal{F}_\beta(\bar \D)$ and $|\varphi_k|_\beta\leqslant |\varphi|_\beta$. Moreover, since  the norm  $\lVert\quad\rVert_\beta$ and the maximal large deviations are both invariant under translation by a constant and $\varphi_k$ is bounded, without loss of generality, we may assume $\varphi_k\geqslant0$, for all $k\in\mathbb{N}$.

In order to use the decay of correlations estimates to obtain the desired maximal large deviations for $\varphi_k$, we introduce the \emph{Perron-Fr\"obenius operator}, also known as \emph{transfer operator}. This operator can be defined in its integral form as the unique linear map $\mathcal{L}$ such that
\begin{equation*}
	\int\left(\mathcal{L}\varphi\right)\psi \,d\bar\nu=\int \varphi\left(\psi\circ  \overline T\right)\,d \bar\nu,
\end{equation*}
for all observables $\varphi\in L^\infty(\bar\nu)$ and $\psi\in L^1(\bar\nu)$. 
Since  $\varphi_k\geqslant0$ and $\mathcal L$ is a positive operator, we get
\begin{equation}\label{eq.nova}
		\int\left|\mathcal{L}^n\varphi_k\right|d\bar\nu = \int \mathcal{L}^n\varphi_kd\bar\nu=\int\varphi_k\cdot1\circ \overline T^nd\bar\nu=\corr_{\bar\nu}\left(\varphi_k,1\circ{\overline{T}}^n\right)=\corr_{ \nu}\left(\varphi_k,1\circ{ {T}}^n\right),
\end{equation}
This last equality follows from the fact that $\varphi_k$ is constant on stable leaves and $\bar\nu=\Theta_*\nu$, where $\Theta$ is the semiconjugacy between $T$ and $\overline T$.

Now we distinguish between two cases, according to the decay of correlations assumed.

\subsubsection{Polynomical case}\label{pol}
Assume that there exists some $C>0$ such that
\begin{equation}\label{dcpoli}
	\text{Cor}_{{\nu}}\left(\psi,\xi\circ {T}^n\right)\leqslant C \lVert\psi\rVert_{\beta}\lVert\xi\rVert_{\beta}n^{-\alpha},
\end{equation}
for every $\xi,\psi\in\mathcal{F}_{\beta}({\D})$. We have in particular  for $\psi=\varphi_k$ and $\xi=1$, 
\begin{equation*}
	\corr_{\nu}\left(\varphi_k,1\circ{T}^n\right)\leqslant C\lVert\varphi_k\rVert_{\beta}n^{-\alpha},
\end{equation*}
which together with~\eqref{eq.nova} yields
\begin{equation*}
		\int\left|\mathcal{L}^n\varphi_k\right|\,d\bar\nu\leqslant C\lVert\varphi_k\rVert_{\beta} n^{-\alpha},
\end{equation*}
Given $q\geqslant 1$, to be chosen later, and recalling that $\lVert\mathcal{L}\varphi_k\rVert_{\infty}\leqslant\lVert\varphi_k\rVert_{\infty},
$
we have
\begin{equation}\label{eq.poly}
	\lVert \mathcal{L}^n\varphi_k\rVert^q_q \leqslant
	\lVert\mathcal{L}^n\varphi_k\rVert_\infty^{q-1}\int\left|\mathcal{L}^n\varphi_k\right|\,d\bar\nu\leqslant C\lVert\mathcal{L}^n\varphi_k\rVert_\infty^{q-1} \lVert\varphi_k\rVert_{\beta} n^{-\alpha}\leqslant C\lVert\varphi_k\rVert_\infty^{q-1}\lVert\varphi_k\rVert_{\beta} n^{-\alpha}.
\end{equation}
Since $\mld_\nu\left(\varphi_k,\epsilon,n\right)=\mld_{\bar\nu}\left(\varphi_k,\epsilon,n\right)$ and  the estimate~\eqref{eq.poly} holds, we now apply \cite[Theorem 3.1]{BS23} to obtain
\begin{equation}\label{eq.mld*}
	\mld_\nu\left(\varphi_k,\epsilon,n\right)\leqslant C_{\alpha,q} \lVert\varphi_k\rVert_{\beta} \lVert\varphi_k\rVert_{\infty}^{2q-1}\epsilon^{-2q}n^{-\alpha},
\end{equation}
for any $q>\max\{1,\alpha\}$, where $C_{\alpha,q}$ is a positive constant only depending on $\alpha$ and $q$. 
Recalling that $ \lVert\varphi_k\rVert_{\infty}\leqslant \lVert\varphi\rVert_{\infty}$ and $\left|\varphi_k\right|_\beta\leqslant\left|\varphi\right|_\beta$, we finally get
\begin{equation*}
	\mld_\nu\left(\varphi_k,\epsilon,n\right) \leqslant C_{\alpha,q}\lVert\varphi\rVert_{\beta}\lVert\varphi\rVert_{\infty}^{2q-1}\epsilon^{-2q}n^{-\alpha}.
\end{equation*}

\subsubsection{Exponential and stretched exponential cases}\label{se.exp}
Assume now that there exist $C, \tau,\theta>0$ such that
\begin{equation*}
	\text{Cor}_{\nu}\left(\psi,\xi\circ {T}^n\right)\leqslant C \lVert\psi\rVert_{\infty}\lVert\xi\rVert_{\infty}e^{-\tau n^\theta},
\end{equation*}
for every $\psi,\xi\in\mathcal{F}_{\beta}({\D})$. As in the previous case, under this assumption we conclude now that there exist constants $C,\tau,\theta>0$  such that 
\begin{equation*}
	\int\left|\mathcal{L}^n\varphi_k\right|d\bar\nu\leqslant C\lVert\varphi_k\rVert_{\beta}e^{-\tau n^\theta}.
\end{equation*}
Since $\ld_\nu\left(\varphi_k,\epsilon,n\right)=\ld_{\bar\nu}\left(\varphi_k,\epsilon,n\right)$ 
applying  \cite[Theorem 2]{AF19}\footnote{Although the statement of this result assumes decay of correlations of $\varphi_k$ against every bounded observable, its proof in fact only requires an (stretched) exponential decay estimate on the moments of $\mathcal{L}^n\varphi_k$.}, we conclude  that
\begin{equation*}
	\ld_\nu\left(\varphi_k,\epsilon,n\right)\leqslant2e^{-\tau_kn^{\theta'}\epsilon^{2\theta'}},
\end{equation*}
where   
\begin{equation*}
\theta'=\theta/(\theta+1),
\quad
\tau_k=C_{\tau,\theta} \lVert\varphi_k\rVert_{\beta}^{-2\theta'}
\end{equation*}
and $C_{\tau,\theta}>0$ is a constant only depending on $\tau,\theta$.
Since $ \lVert\varphi_k\rVert_{\infty}\leqslant \lVert\varphi\rVert_{\infty}$ and $\left|\varphi_k\right|_\beta\leqslant\left|\varphi\right|_\beta$, we may write 
\begin{equation}\label{ldexp}
	\ld_{\nu}\left(\varphi_k,\epsilon,n\right)\leqslant 2e^{-\tau' n^{\theta'}\epsilon^{2\theta'}}.
\end{equation} 
with
\begin{equation*}
	\tau'= C_{\tau,\theta} \lVert\varphi\rVert_{\beta}^{-2\theta'}.
\end{equation*}
 This yields large deviations estimates for  $\varphi_k$ not depending on $k$. These estimates can, in turn, be used to derive maximal large deviations bounds. Indeed,
\begin{equation} \label{eq.sea}
	\begin{aligned}
		\mld_\nu(\varphi_k,\epsilon,n)&=\nu\left(\sup_{j\geqslant n}\left|\frac{1}{j}\sum_{i=0}^{j-1}\varphi_k\circ T^i-\int\varphi_k\right|>\epsilon\right)\\&\leqslant\sum_{j\geqslant n}\ld_{{\nu}}(\varphi_k,\epsilon,j)\\
		&\leqslant \sum_{j\geqslant n} 2e^{-\tau'j^{\theta'}\epsilon^{2\theta'}}\\
		&
		\leqslant C_{\tau,\theta,\varphi}e^{-\tau' n^{\theta'}\epsilon^{2\theta'}},
	\end{aligned}
\end{equation}
where $C_{\tau,\theta,\varphi}>0$ is a constant only depending  on $\tau,\theta$ and $\varphi$.

We remark that in \cite[Remark 3.1]{CDM22}, the authors deduce a sharper estimate for (stretched) exponential large deviations. However, this   result cannot be used in our context, since it needs information on the tails of recurrence times. In fact, \cite{CDM22} gives an improvement in the context of return tails imply large deviations, but not in the context of decay of correlations imply large deviations, being the last one the case we consider.

\subsection{Maximal large deviations for the original observable}\label{sub.original}
We are now ready to conclude the proof of Theorem \ref{main3}, by obtaining estimates for $\mld_{{\nu}}\left(\psi,\epsilon,n\right)$. At this point the rates of decay are not important to specify, so we denote them by $r(n)$. In any case, the following argument works for polynomial, stretched exponential and exponential cases. Recall also that the value of $n$ is fixed. Indeed, let 
\begin{equation*}
	A_k=\left\{\sup_{j\geqslant n}\left|\frac{1}{j}\sum_{i=0}^{j-1}\varphi_k\circ{T}^i-\int\varphi_kd{\nu}\right|>\epsilon\right\}
\end{equation*}
and 
\begin{equation*}
	B=\left\{\sup_{j\geqslant n}\left|\frac{1}{j}\sum_{i=0}^{j-1}\varphi\circ{T}^i-\int\varphi d{\nu}\right|>\epsilon\right\}
\end{equation*}
Setting  $A=\liminf  A_k$, we have ${\nu}(A)\leqslant\liminf{\nu}(A_k)$. Since $${\nu}(A_k)=\mld_{\nu}\left(\varphi_k,\epsilon,n\right)\leqslant C r(n)$$ and $C$ does not depend on $k$, we deduce that $\liminf{\nu}(A_k)\leqslant C r(n)$, and therefore ${\nu}(A)\leqslant C r(n)$. 
Now we check that $B\subset A$. Indeed, let $x\in B$. This means that 
\begin{equation*}
	\sup_{j\geqslant n}{\left|\frac{1}{j}\sum_{i=0}^{j-1}\varphi\circ{T}^i(x)-\int\varphi d{\nu}\right|>\epsilon}.
\end{equation*}
By Lemma \ref{uniconv}, $\varphi_k$ converges ${\nu}$ almost everywhere to $\varphi$, and thus, using Dominated Convergence Theorem, we deduce that 
\begin{equation*}
	\lim\limits_{k\to+\infty}\left(\varphi_k-\int\varphi_kd \nu\right)=\varphi-\int\varphi d \nu\quad \text{  $\nu$-a.e.}.
\end{equation*}Hence, for almost every $x$ and every $j\geqslant n$, 
\begin{equation}\label{domconv}
	\lim\limits_{k\to+\infty}\left|\frac{1}{j}\sum_{i=0}^{j-1}\varphi_k\circ{T}^i(x)-\int\varphi_kd{\nu}\right|= 	\left|\frac{1}{j}\sum_{i=0}^{j-1}\varphi\circ{T}^i(x)-\int\varphi d{\nu}\right|
\end{equation}
This means that, if $x\in B$ we must have, for every $k$ large enough,
\begin{equation*}
	\sup_{j\geqslant n}\left|\frac{1}{j}\sum_{i=0}^{j-1}\varphi_k\circ{T}^i(x)-\int\varphi_kd{\nu}\right|>\epsilon,
\end{equation*}
otherwise \eqref{domconv} would not be possible. But that means that $x\in A$. Hence, $B\subset A$ and, therefore,
\begin{equation*}
	\mld_{{\nu}}\left(\varphi,\epsilon,n\right)= {\nu}(B)\leqslant {\nu}(A)\leqslant Cr(n),
\end{equation*}
which gives us the desired large deviations estimates, concluding the proof of Theorem \ref{main3}.

\section{Maximal large deviations in the original system}\label{lift}
Here we   prove Theorem~\ref{Maximal large deviations}.
 As in the proof of Theorem \ref{main3}, we will not study directly the observable $\varphi$. We lift it to the tower and then consider another observable, with the same rates of maximal large deviations as the lifted one, to which we will apply Theorem \ref{main3}.
\subsection{Lifting to a tower}\label{sub.lift}
Consider  a dynamical system $f\colon M \to M$ that admits a Young tower $T\colon \D \to \D$ and a H\"older observable $\varphi:M\to\mathbb R$.  Consider  also $\pi$ the semiconjugacy between $T$ and $f$ given by~\eqref{eq.conjuga},  $\nu$  the unique ergodic SRB measure for $T$ and $\mu = \pi_*\nu$  the unique ergodic SRB measure for~$f$ such that $\mu(\Lambda) > 0$.
For clarity, we encourage the reader to keep in mind the following commutative diagram, where $\tilde{\varphi}$ denotes the lift of the observable $\varphi$ to the tower, defined by $\tilde{\varphi} = \varphi \circ \pi$.
\[\begin{tikzcd}
	\D && \D \\
	\\
	M && M && \mathbb{R}
	\arrow["\pi", from=1-1, to=3-1]
	\arrow["f", from=3-1, to=3-3]
	\arrow["\pi"', from=1-3, to=3-3]
	\arrow["T", from=1-1, to=1-3]
	\arrow["\varphi", from=3-3, to=3-5]
	\arrow["\tilde \varphi", from=1-3, to=3-5]
\end{tikzcd}\]
With this in mind, we begin by lifting our problem to the tower $T$ and see that the large deviations estimates for $\varphi$ and $\tilde{\varphi}$ are of the same type.

\begin{lemma}\label{ld1}
	$ 
	\mld_\mu\left(\varphi,\epsilon,n\right)=\mld_{\nu}(\tilde{\varphi},\epsilon,n).
	$ 
\end{lemma}

\begin{proof}
	Without loss of generality, we may assume that $\int \varphi d\mu=0$. Since $\pi_*{\nu}=\mu$, we have
	\begin{equation*}
		\begin{aligned}
			\mld_{\mu}\left(\varphi, \epsilon, n\right)&=\mu\left\{\sup_{j\geqslant n}\left| \frac{1}{j}\sum_{i=0}^{j-1}\varphi\circ f^i\right|>\epsilon\right\}\\
			&=\pi_*{\nu}\left\{\sup_{j\geqslant n}\left|\frac{1}{j}\sum_{i=0}^{j-1}\varphi\circ f^i\right|>\epsilon\right\}\\
			&={\nu}\left(\pi^{-1}\left\{\sup_{j\geqslant n}\left|\frac{1}{j}\sum_{i=0}^{j-1}\varphi\circ f^i\right|>\epsilon\right\}\right)\\
			&={\nu}\left\{\sup_{j\geqslant n}\left|\frac{1}{j}\sum_{i=0}^{j-1}\varphi\circ f^i \circ \pi\right|>\epsilon\right\}\\
			&={\nu}\left\{\sup_{j\geqslant n}\left|\frac{1}{j}\sum_{i=0}^{j-1}\varphi\circ \pi \circ {T}^i\right|>\epsilon\right\}\\&= \mld_{{\nu}}(\tilde{\varphi},\epsilon,n).
		\end{aligned}
	\end{equation*}
This completes the proof.
\end{proof}

\begin{rmk}\label{rmk.semi}
In the proof of the previous lemma we have only used that $\pi:\Delta\to M$ is a semiconjugacy between the dynamical systems $(T,\nu)$ and $(f,\mu)$, for which $\pi_*\nu=\mu$.
\end{rmk}

\subsection{Maximal large deviations for the lifting}
Instead of working directly with the lifted observable $\tilde{\varphi}$, we use the following decomposition presented in \cite[Lemma~3.2]{MN05},
\begin{equation}\label{MN}
	\tilde{\varphi}=\psi+\chi-\chi\circ T,
\end{equation}
where $\psi\colon{\D}\to\mathbb{R}$ belongs in $\mathcal{F}_{\beta'}(\D)$, with   $\beta'=\beta^{\eta/2}$, and   $\chi\colon{\D}\to\mathbb{R}$ is bounded. Moreover, 
\begin{equation}\label{eq.normas}
\|\chi\|_\infty \leqslant C\|\varphi\|_{\mathcal H_\eta}\qand \|\psi\|_\beta \leqslant C\|\varphi\|_{\mathcal H_\eta}.
\end{equation}
 Since $\psi\in \mathcal{F}_{\beta'}(\D)$  the conclusions of Theorem~\ref{main3} hold, and so 
\begin{equation}\label{eq.um}
	\mld_{\nu}(\psi,\epsilon,n)\leqslant C_{\alpha,q}   \lVert\psi\rVert_{\beta'}\lVert\psi\rVert_{\infty}^{2q-1}\epsilon^{-2q}n^{-\alpha}
\end{equation}
and 
\begin{equation}\label{eq.dois}
	 \mld_{\nu}(\psi,\epsilon,n)\leqslant C_{\tau,\theta,\psi}e^{-\tau' n^{\theta'}\epsilon^{2\theta'}}
\end{equation}
with  constants as in Theorem \ref{main3}.

Now, recalling Lemma~\ref{ld1}, we just have to check that $\mld_{\nu}(\tilde{\varphi},\epsilon,n)$ and $\mld_{\nu}\left(\psi,\epsilon,n\right)$ have the same type of convergence to 0, as $n\to\infty$. Without loss of generality, we may assume that $\int\tilde{\varphi}d{\nu}=0$. This way, we have  
\begin{equation}\label{eq.int}
	\int\psi d{\nu}+\int \chi d{\nu}-\int\chi\circ{T} d{\nu}=0.
\end{equation}
Since the measure ${\nu}$ is invariant under ${T}$, then $\int \chi d{\nu}=\int\chi\circ{T} d{\nu}$, and the previous equation becomes $\int\psi d{\nu}=0$. This means that
\begin{equation*}
	\mld_{{\nu}}\left(\psi, \epsilon, n\right)={\nu}\left\{\sup_{j\geqslant n}\left| \frac{1}{j}\sum_{i=0}^{j-1}\psi\circ {T}^i\right|>\epsilon\right\}.
\end{equation*}
Since, for every $n\in\mathbb{N}$, we have $$\left|\frac{1}{n}\sum_{i=0}^{n-1}\tilde{\varphi}\circ{T}^i-\frac{1}{n}\sum_{i=0}^{n-1}\psi\circ{T}^i\right|\leqslant\frac{2\lVert\chi\rVert_\infty}{n},$$ by taking $n$ sufficiently large so that $2\lVert\chi\rVert_{\infty}/n\leqslant\epsilon/2$, we deduce that 
\begin{equation}\label{eq.comld}
	\mld_{\nu}(\tilde{\varphi},\epsilon,n)\leqslant\mld_{\nu}\left(\psi,\epsilon/2,n\right).
\end{equation} 
Finally, using \eqref{eq.normas}-\eqref{eq.comld} and Lemma~\ref{ld1}, we  finish  the proof of  Theorem~\ref{Maximal large deviations}.

\section{Young structures and recurrence times}\label{proof2}


Here we prove Theorems~\ref{main2} and \ref{main}. 
We first prove  Theorem \ref{main2} and leave the proof of Theorem~\ref{main} for the end of this section, where it will be obtained as a consequence of Theorem \ref{main2}  and Theorem~\ref{Maximal large deviations}. 
The proof of \cite[Proposition 9.1]{ADL17} provides $N\geqslant 1$ such that 
\begin{equation}\label{neg}
\int \log \|(Df^{N}|_{E_{x}^{cu}})^{-1}\| d\mu < 0
\end{equation}
and  \(  (f^{N}, \mu)  \) has at most \(  N  \)
ergodic components. By \eqref{neg}, at least one of these ergodic components, whose support
 we denote by \(  \Sigma \),
satisfies
\(
    \int_{\Sigma}\log \|(Df^{N}|{E_{x}^{cu}})^{-1}\|d\mu < 0.
\)
Hence, using  Birkhoff's Ergodic Theorem, we deduce that 
\(  f^{N}  \) is \emph{non-uniformly expanding along~\(  E^{cu}  \)}
almost everywhere:
\begin{equation}\label{eq.integral} 
\lambda=\lim_{n\rightarrow\infty}\frac{1}{n}\sum_{j=0}^{n-1}
\log\|(Df^N|{E_{f^{Nj}(x)}^{cu}})^{-1}\| =\int_{\Sigma}\log\|(Df^N|{E_{x}^{cu}})^{-1}\|d\mu <0,
 \end{equation}
 for $\mu$ almost every~\(x\in \Sigma\).  It easily follows that  
 \begin{equation*}
 	\mathcal{E}(x)=\min\left\{k\in\mathbb{N}\colon \frac{1}{n}\sum_{i=0}^{n-1}\phi\left(f^{Ni}(x)\right)\leqslant \frac{\lambda}{2}, \forall n\geqslant k\right\}
 \end{equation*}
 is defined and finite $\mu$ almost everywhere. The proof of Theorem~\ref{main2} will be deduced by mean of \cite[Theorem~A]{AP10}, in the polynomial case, and  \cite[Theorem~A]{AL15}, in the (stretched) exponencial case. We just need
 to see that $\leb_{\gamma_0}\left\{\mathcal{E}>n\right\}$ decays with at most the same rate of $\mld_{\mu}(\phi,\epsilon,n)$, for some (hence, for all) unstable disk $\gamma_0\in \Gamma^u$. This will be achieved in Lemma~\ref{AFLV3.2} below. 
 
Using~\eqref{eq.integral} and  \cite[Theorem D]{ADL17} applied  to $f^N$, we obtain a Young structure   $\Lambda\subset \Sigma$, given by a set of stable disks $\Gamma^s$ and a set of unstable disks $\Gamma^u$ as in Subsection~\ref{se.young}. Moreover, considering the associated Young tower map $T:\Delta\to\Delta$ and the respective projection $\pi:\Delta\to M$, we have
 $\mu=\pi_*\nu$, where $\nu$ is the unique SRB measure for $T$; recall Subsection~\ref{sub.lift}. 
  Consider now the H\"older observable
\begin{equation}\label{eq.phin}
		\phi=\log\lVert \left(Df^N|_{E^{cu}_x}\right)^{-1}\rVert
	\end{equation} 
 and $\tilde\phi=\phi\circ\pi$ the lift to the tower. It follows from Lemma~\ref{ld1} that
  \begin{equation*}
	\mld_\mu\left(\phi,\epsilon,n\right)=\mld_{\nu}(\tilde{\phi},\epsilon,n).
	 \end{equation*}
As observed in~\eqref{MN}, there exist $\psi\in \mathcal{F}_{\beta^{\eta/2}}(\D)$   which is constant on stable leaves and $\chi$ bounded such that
\begin{equation*}\label{MN2}
	\tilde{\phi}=\psi+\chi-\chi\circ T.
\end{equation*}
Since $\psi$ is constant on stable disks, we may also think of it as an observable on the quotient tower $\bar\Delta$ introduced in Subsection~\ref{se.quotient}. 
Moreover, as in  \eqref{eq.int}
$$\int\psi d\bar\nu=\int\psi d \nu=\int\tilde \phi d\nu= \int \phi d\mu.$$ 
Therefore, it is no restriction to assume that these integrals are all equal to zero.
It also follows that for every $n\in\mathbb{N}$ 
  $$
  \left|\frac{1}{n}\sum_{i=0}^{n-1}{\tilde\phi}\circ{ T}^i - \frac{1}{n}\sum_{i=0}^{n-1}\psi\circ{ T}^i \right|\leqslant\frac{2\lVert\chi\rVert_\infty}{n},
  $$ 
 which then implies, as in \eqref{eq.comld},
\begin{equation}\label{eq.bounddif}
	\mld_{ {\nu}}(\tilde{\phi},\epsilon,n) \geqslant \mld_{ {\nu}}(\psi,3\epsilon/2,n).
\end{equation}
for $n$ sufficiently large.   As in   Lemma~\ref{ld1} (recall Remark~\ref{rmk.semi}), we obtain
	$$\mld_{{\nu}}(\psi,3\epsilon/2,n)=\mld_{\bar{\nu}}(\psi,3\epsilon/2,n).$$
%
%
Now,   note that by the construction of the quotient tower in Subsection~\ref{se.quotient}, we have that the level zero $\bar\Delta_0$ of the quotient tower is an unstable disk $\gamma_0\in\Gamma^u$. Consider  $\nu_0=\bar\nu\vert_{\gamma_0}$. Therefore,
\begin{equation*}
\begin{aligned}
	\mld_{\bar{\nu}}(\psi,3\epsilon/2,n)	&=\bar\nu\left\{\sup_{j\geqslant n}\left|\frac{1}{j}\sum_{i=0}^{j-1}\psi\circ \overline T^i \right|>\frac32\epsilon\right\}\\
	&\geqslant \nu_0\left\{\sup_{j\geqslant n}\left|\frac{1}{j}\sum_{i=0}^{j-1}\psi\circ \overline T^i \right|>\frac32\epsilon\right\}\\
	&= \nu_0\left\{\sup_{j\geqslant n}\left|\frac{1}{j}\sum_{i=0}^{j-1}\psi\circ \ T^i \right|>\frac32\epsilon\right\},
\end{aligned}
\end{equation*}
where in the last equality we used the fact that   $T|_{\gamma_0}=\overline{T}|_{\gamma_0}$.
\begin{lemma}
	There exists a constant $c>0$ such that for $n$ sufficiently large
	\begin{equation*}
  \nu_0\left\{\sup_{j\geqslant n}\left|\frac{1}{j}\sum_{i=0}^{j-1}\psi\circ T^i \right|>\frac32\epsilon\right\}\geqslant 	c	\leb_{\gamma_0}\left\{\sup_{j\geqslant n}\left|\frac{1}{j}\sum_{i=0}^{j-1}\phi\circ f^{Ni} \right|>\frac52\epsilon\right\}.
	\end{equation*}
\end{lemma}

\begin{proof}
	 By the proof of \cite[Theorem 3.24]{A20}, there is a constant $c>0$ such that ${d\nu_0/d\leb_{\gamma_0}>c}$. Therefore,
	 	\begin{equation*}
		\begin{aligned}
	 \nu_0\left\{\sup_{j\geqslant n}\left|\frac{1}{j}\sum_{i=0}^{j-1}\psi\circ T^i\right|>\frac32\epsilon\right\}
	\geqslant c\leb_{\gamma_0}\left\{\sup_{j\geqslant n}\left|\frac{1}{j}\sum_{i=0}^{j-1}\psi\circ T^i\right|>\frac32\epsilon\right\}
		\end{aligned}
	\end{equation*}
Moreover, by \eqref{eq.bounddif}, we deduce that, for $n$ sufficiently large
	\begin{equation*}
		\leb_{\gamma_0}\left\{\sup_{j\geqslant n}\left|\frac{1}{j}\sum_{i=0}^{j-1}\psi\circ  T^i \right|>\frac32\epsilon\right\} \geqslant \leb_{\gamma_0}\left\{\sup_{j\geqslant n}\left|\frac{1}{j}\sum_{i=0}^{j-1}\tilde \phi \circ {T}^i \right|>\frac52\epsilon\right\}.
	\end{equation*}
Now, for any $x \in\gamma_0$ and any $i\in\mathbb{N}$,
	\begin{equation*}
		\begin{aligned}
			\tilde{\phi}\circ T^i(x,0)&=\phi\circ\pi\circ {T}^i(x,0)\\
			&=\phi\circ f^{Ni}\circ\pi(x,0)\\
			&=\phi\circ f^{Ni}(x).
		\end{aligned}
	\end{equation*}
	Hence,
	\begin{equation*}
		\leb_{\gamma_0}\left\{\sup_{j\geqslant n}\left|\frac{1}{j}\sum_{i=0}^{j-1}\tilde \phi \circ {T}^i \right|>\frac52\epsilon\right\}=\leb_{\gamma_0}\left\{\sup_{j\geqslant n}\left|\frac{1}{j}\sum_{i=0}^{j-1}\phi \circ f^{Ni} \right|>\frac52\epsilon\right\},
	\end{equation*}
	thus yielding the expected  result.
\end{proof}

From what we have seen above, we conclude  that, for $n\in\mathbb{N}$ sufficiently large,
	\begin{equation}\label{eq.LebMLD}
		\mld_\mu(\phi,\epsilon,n)\geqslant  c\leb_{\gamma_0}\left\{\sup_{j\geqslant n}\left|\frac{1}{j}\sum_{i=0}^{j-1}\phi \circ f^{Ni} \right|>\frac52\epsilon\right\}.
	\end{equation}
The next lemma concludes the proof of Theorem~\ref{main2}.

\begin{lemma}\label{AFLV3.2}
	For $n\in\mathbb{N}$ sufficiently large,
	\begin{equation*}
		\mld_{\mu}(\phi,\epsilon,n)\geqslant c \leb_{\gamma_0}\left\{\mathcal{E}>n\right\}
	\end{equation*}
\end{lemma}

\begin{proof}
	Set
	\begin{equation*}
		{S}_n\phi(x) =\left|\frac{1}{n}\sum_{i=0}^{n-1}\phi\left(f^{Ni}(x)\right) \right|,
	\end{equation*}
	By Birkhoff's Ergodic Theorem, ${S}_n\phi(x)$ converges to zero, for $\mu$ almost every $x$. Set also, for $\epsilon>0$,
	\begin{equation*}
		N_\epsilon(x)= \min\left\{N\in\mathbb{N}\colon {S}_n\phi(x)\leqslant\epsilon, \forall n\geqslant N\right\}.
	\end{equation*}	
	Since
	\begin{equation}\label{mldineq}
		\left\{N_\epsilon>n\right\}\subset \left\{\sup_{j\geqslant n} {S}_j\phi>\epsilon\right\},
	\end{equation}
	we deduce that
	\begin{equation*}
	\leb_{\gamma_0}\left\{\sup_{j\geqslant n} {S}_j\phi>\epsilon\right\} \geqslant	\leb_{\gamma_0}\left\{N_\epsilon>n\right\}.
	\end{equation*}
	Together with \eqref{eq.LebMLD}, we conclude that, for $n$ sufficiently large,
	\begin{equation*}
	\mld_\mu(\phi,\epsilon,n) \geqslant c	\leb_{\gamma_0}\left\{N_{5\epsilon/2}>n\right\}.
	\end{equation*}
	Finally,  setting  $\epsilon = -\lambda/5$, we have
	\begin{equation*}
		\left\{\mathcal{E}>n\right\}\subseteq \left\{N_\epsilon>n\right\},
	\end{equation*}
and so
		\begin{equation*}
		\mld_{\mu}(\phi,\epsilon,n)\geqslant c\leb_{\gamma_0}\left\{\mathcal{E}>n\right\},
	\end{equation*}
	for $n$ sufficiently large.
\end{proof} 

%
%
%

 \begin{rmk} 
The inequality~\eqref{mldineq} constitutes the principal motivation for considering maximal large deviations in this context. Indeed, were we to rely on standard large deviations instead, we would be compelled to compare the set $\left\{N_\epsilon > n\right\}$ with families of the form $\left\{S_n \phi > \epsilon\right\}$, and in place of the inclusion~\eqref{mldineq} we would merely obtain
\[
\left\{N_\epsilon > n\right\} \subset \bigcup_{\ell \geqslant n} \left\{S_\ell \phi > \epsilon\right\}.
\]
In this alternative framework, the conclusion of Lemma~\ref{AFLV3.2} would then become
\begin{equation*}
	\sum_{\ell \geqslant n} \ld_{\mu}(\phi,\epsilon,\ell) \geqslant c\, \leb_{\gamma_0}\left(\left\{\mathcal{E} > n\right\}\right),
\end{equation*}
for some positive constant $c$.
While this modification does not significantly affect the exponential and stretched exponential regimes, in the polynomial case it would result in a strictly weaker estimate, namely
\[
\leb_{\gamma_0}\left(\left\{\mathcal{E} > n\right\}\right) \leqslant C n^{-\alpha+1},
\]
which is suboptimal. This distinction is precisely what sets apart the proofs of \cite[Theorem~3.1]{AFL11} and \cite[Theorem~4.6]{BS23}, the latter yielding sharper decay rates for the tail distribution of the return time.
 \end{rmk}

We conclude by establishing the proof of Theorem~\ref{main}. More precisely, we demonstrate that the dynamical system under consideration admits a Young structure -- although not necessarily the same one presupposed in the assumptions of Theorem~\ref{main} -- whose recurrence times exhibit asymptotic behaviour of the same nature as the decay of correlations and the maximal large deviations.
To that end, consider  the observable
\begin{equation*}
	\phi := \log \left\lVert \left(Df^N|_{E^{cu}_x} \right)^{-1} \right\rVert.
\end{equation*}
Since the map $f$ is a $C^{1+\eta}$ diffeomorphism, it follows that the function $\phi$ is $\eta$-H\"older continuous. Consequently, by virtue of Theorem~\ref{Maximal large deviations}, we obtain the following maximal large deviations estimates for $\phi$: in the case of polynomial decay,
\begin{equation*}
	\mld_{\mu}(\phi, \epsilon, n) \lesssim n^{-\alpha},
\end{equation*}
and in the case of (stretched) exponential decay,
\begin{equation*}
	\mld_{\mu}(\phi, \epsilon, n) \lesssim e^{-\tau n^\theta}.
\end{equation*}
In view of these two estimates, the conclusion of Theorem~\ref{main} follows directly from Theorem~\ref{main2}.

\bibliographystyle{acm}

\end{document}